\documentclass{amsart}

\usepackage{amsmath,amssymb,amsfonts,amstext,amsthm}
\usepackage{url}
\usepackage{graphicx}
\input xy
\xyoption{all}
\usepackage[all]{xy}

\graphicspath{{images/}{../images/}}
\usepackage{subfiles}
\usepackage{tikz}

\newcommand{\nwc}{\newcommand}

\nwc{\aaa}{\mathcal{F}}
\nwc{\aap}{\mathcal{F}_{P}}
\nwc{\al}{\alpha}

\nwc{\C}{\mathbb{C}}
\nwc{\cb}{\overline{C}}
\nwc{\ccc}{\mathfrak{c}}
\nwc{\ch}{\widehat{C}}
\nwc{\cin}{\textbf{(v)}}
\nwc{\cl}{C'}
\nwc{\cp}{\mathcal{C}_{P}}
\nwc{\cpll}{\mathfrak{c}_{P'}}
\nwc{\ct}{\widetilde{C}}

\nwc{\dd}{\mathcal{L}}
\nwc{\ddd}{\mathfrak{d}}
\nwc{\ddl}{\mathcal{L}'}
\nwc{\dlp}{\delta_{P}}
\nwc{\doi}{\textbf{(ii)}}

\nwc{\enq}{$$}

\nwc{\fl}{\flushleft}
\nwc{\fff}{\mathcal{F}}
\nwc{\ffp}{\mathcal{F}_{P}}
\nwc{\ffq}{\mathcal{F}_{Q}}
\nwc{\ffl}{\mathcal{F}'}

\nwc{\G}{\mathcal{G}}
\nwc{\Ga}{\Gamma}
\nwc{\gtl}{\widetilde{g}}

\nwc{\hra}{\hookrightarrow}
\nwc{\hua}{h^{1}(C,\aaa )}

\nwc{\kk}{{\rm K}}

\nwc{\llb}{\mathcal{L}}

\nwc{\mb}{\mathbb}
\nwc{\mc}{\mathcal}
\nwc{\mm}{\mathfrak{m}}
\nwc{\mmp}{\mathfrak{m}_{P}}
\nwc{\mpd}{\mathfrak{m}_{P}^{2}}

\nwc{\nn}{\mathbb{N}}

\nwc{\ob}{\overline{\mathcal{O}}}
\nwc{\obr}{\mathcal{O}^*}
\nwc{\obp}{\overline{\mathcal{O}}_P}
\nwc{\och}{\mathcal{O}_{\hat{C}}}
\nwc{\oh}{\hat{\mathcal{O}}}
\nwc{\ohp}{\hat{\mathcal{O}}_{P}}
\nwc{\ol}{\mathcal{O}'}
\nwc{\oma}{\Omega (\mathfrak{a})}
\nwc{\omo}{\Omega (\mathcal{O})}
\nwc{\oo}{\mathcal{O}}
\nwc{\op}{\mathcal{O}_P}
\nwc{\opc}{\mathcal{O}_{P,C}}
\nwc{\oph}{\hat{\mathcal{O}}_{P}}
\nwc{\opl}{\mathcal{O}_{P}'}
\nwc{\oplc}{\mathcal{O}_{P,C}'}
\nwc{\opll}{\mathcal{O}_{P'}}
\nwc{\opt}{\tilde{\mathcal{O}}_{P}}
\nwc{\optt}{{\mathcal{O}}_{\tilde{P}}}
\nwc{\oq}{\mathcal{O}_{Q}}
\nwc{\oqt}{\tilde{\mathcal{O}}_{Q}}
\nwc{\ot}{\widetilde{\mathcal{O}}}
\nwc{\overop}{\bar{\oo}_{P}}

\nwc{\pb}{\overline{P}}
\nwc{\pbb}{P^*}
\nwc{\pbi}{\overline{P_{i}}}
\nwc{\pbr}{\overline{P_{r}}}
\nwc{\pgmd}{\mathbb{P}^{g+2}}
\nwc{\pgmu}{\mathbb{P}^{g+1}}
\nwc{\ph}{\hat{P}}
\nwc{\pp}{\mathbb{P}}
\nwc{\prv}{\noindent\textbook{Proof}:}
\nwc{\pt}{\widetilde{P}}
\nwc{\ptl}{\tilde{P}}
\nwc{\pum}{\mathbb{P}^{1}}

\nwc{\qh}{\hat{Q}}
\nwc{\qtl}{\tilde{Q}}
\nwc{\qua}{\textbf{(iv)}}

\nwc{\ra}{\rightarrow}
\nwc{\rh}{\hat{R}}

\nwc{\sei}{\textbf{(vi)}}
\nwc{\sep}{\beq\ast\ \ast\ \ast\enq}
\nwc{\sig}{\sigma}
\nwc{\Sig}{\Sigma}
\nwc{\ssp}{S_{P}}
\nwc{\sss}{{\rm S}}

\nwc{\tre}{\textbf{(iii)}}

\nwc{\um}{\textbf{(i)}}

\nwc{\vpb}{v_{\overline{P}}}
\nwc{\vtxp}{\widetilde{V}_{x,P}}
\nwc{\vxp}{V_{x,P}}

\nwc{\wh}{\hat{\omega}}
\nwc{\whp}{\hat{\omega}_{P}}
\nwc{\woch}{\omega\cdot\mathcal{O}_{\hat{C}}}
\nwc{\woh}{\omega\cdot\hat{\mathcal{O}}}
\nwc{\ww}{\omega}
\nwc{\wwb}{\omega^*}
\nwc{\wwct}{\omega _{\widetilde{C}}}
\nwc{\wwh}{\widehat{\omega}}
\nwc{\wwhp}{\widehat{\omega}_P}
\nwc{\wwp}{\omega _{P}}
\nwc{\wwt}{\widetilde{\omega}}
\nwc{\wwtp}{\widetilde{\omega}_P}

\nwc{\zz}{\mathbb{Z}}

\newtheorem{coro}{Corollary}[section]

\newtheorem{prop}[coro]{Proposition}

\newtheorem{rem}[coro]{Remark}

\newtheorem{thm}[coro]{Theorem}

\let \fl=\flushleft

\let \ga=\gamma
\let \sub=\subset

\let \al=\alpha

\begin{document}

\title{${\rm K}$-weight bounds for $\ga$-hyperelliptic semigroups}
\author{Ethan Cotterill}
\address{Instituto de Matem\'atica, Universidade Federal Fluminense, Rua Prof Waldemar de Freitas, S/N, Campus do Gragoat\'a, CEP 24.210-201, Niter\'oi, RJ, Brazil}
\email{cotterill.ethan@gmail.com}


\author{Renato Vidal Martins}
\address{Departamento de Matem\'atica, ICEx, UFMG
Av. Ant\^onio Carlos 6627,
30123-970 Belo Horizonte MG, Brazil}
\email{renato@mat.ufmg.br}

\maketitle

\begin{abstract}
In this note, we show that {\it $\ga$-hyperelliptic} numerical semigroups of genus $g \gg \ga$ satisfy a refinement of a well-known characteristic weight inequality due to Torres. The refinement arises from substituting the usual notion of weight by an alternative version, the $K$-weight, which we previously introduced in the course of our study of unibranch curve singularities. 
\end{abstract}

\section{$K$-weights of numerical semigroups}
Let ${\rm S} \sub \mb{N}$ denote a numerical semigroup of genus $g$. Recall that this means that the complement $G_{\rm S}= \mb{N} \setminus {\rm S}$ is of cardinality $g$; say
\[
G_{\rm S}= \{\ell_1, \dots, \ell_g\}
\]
where $\ell_i < \ell_j$ whenever $1 \leq i< j \leq g$. Following \cite{CFM}, we define the {\it ${\rm K}$-weight} of ${\rm S}$ to be the quantity
\[
W_{\rm K}:= \sum_{i=1}^{g-1} (\ell_i-i) +g-1.
\]
The definition of the ${\rm K}$-weight was motivated by the study of unibranch complex curve singularities. It is also closely related to a more familiar notion of weight, which we will call the ${\rm S}$-weight, namely:
\[
W_{\rm S}:= \sum_{i=1}^{g} (\ell_i-i).
\]
The ${\rm S}$-weight emerges naturally in the study of Weierstrass semigroups of (points of) {\it smooth} complex curves. It is not hard to show that $W_{\rm K}=W_{\rm S}$ whenever ${\rm S}$ is the value semigroup of a Gorenstein singularity. In purely combinatorial terms, the ${\rm K}$- and ${\rm S}$-weights agree if and only if ${\rm S}$ is a {\it symmetric} semigroup.

\section{$K$-weights of $\ga$-hyperelliptic semigroups}
Now let $\ga \geq 0$ be an integer. Recall from \cite{To2} that ${\rm S}$ is {\it $\ga$-hyperelliptic} if it satisfies the following conditions:
\begin{enumerate}
\item ${\rm S}$ has $\ga$ even elements in $[2,4\ga]$; and
\item The $(\ga+1)$st positive element of ${\rm S}$ is $4\ga+2$.
\end{enumerate}
In \cite{CFM}, we conjectured that when $g \gg \ga$, any $\ga$-hyperelliptic numerical semigroup ${\rm S}$ satisfies
\begin{equation}\label{gamma_hyperelliptic_weights}
\binom{g-2\ga}{2}+ 2\ga \leq W_{\rm K} \leq \binom{g-2\ga}{2}+ 2\ga^2.
\end{equation}
The inequalities \eqref{gamma_hyperelliptic_weights} are not far-removed from the inequalities
\begin{equation}\label{gamma_hyperelliptic_weights_torres}
\binom{g-2\ga}{2} \leq W_{\rm S} \leq \binom{g-2\ga}{2}+ 2\ga^2
\end{equation}
proved by Torres in \cite{To,To2}. As the disparity between upper and lower bounds in \eqref{gamma_hyperelliptic_weights} is in general smaller than that of \eqref{gamma_hyperelliptic_weights_torres}, we regard the ${\rm K}$-weight inequalities as a refinement of the ${\rm S}$-weight inequalities.

\medskip
In this note, we will prove the ${\rm K}$-weight inequalities \eqref{gamma_hyperelliptic_weights} are satisfied by any $\ga$-hyperelliptic semigroup of sufficiently large genus. In doing so, we adapt the geometric interpretation of the ${\rm S}$-weight given in \cite{BdM}. Namely, each numerical semigroup may be represented as a Dyck path $\tau=\tau({\rm S})$ on a $g \times g$ square grid $\Ga$ with axes labeled by $0,1, \dots, g$. Each path starts at $(0,0)$, ends at $(g,g)$, and has unit steps upward or to the right. The $i$th step of $\tau$ is up if $i \notin {\rm S}$, and is to the right otherwise. The weight $W_{\rm S}$ of ${\rm S}$ is then equal to the total number of boxes in the Young tableau $T_{{\rm S}}$ traced by the upper and left-hand borders of the grid and the Dyck path $\tau$. The contribution of each gap $\ell$ of ${\rm S}$ to $W_{\rm S}$ is then computed by the number of boxes inside the grid and to the left of the corresponding path edge.

\begin{thm}\label{ineq_thm}
Fix a choice of non-negative integer $\ga$, and let ${\rm S}$ denote a $\ga$-hyperelliptic semigroup of genus $g \ge 2\ga+1$. The $K$-weight of ${\rm S}$ then satisfies the inequalities \eqref{gamma_hyperelliptic_weights}.
\end{thm}

\begin{proof}
Suppose ${\rm S}$ is a $\ga$-hyperelliptic semigroup of genus $g$. For the calculation of the ${\rm K}$-weight of ${\rm S}$, the largest gap $\ell_g$ i $\mb{N} \setminus {\rm S}$ is irrelevant, so we focus on the subdiagram of $\Ga$ given by omitting the uppermost row of boxes in the grid, and on the corresponding subtableau $T_{\rm S}$, which we will denote by $T_{\rm K}$.  

\medskip
It is well-known (and follows easily from the semigroup structure in any case) that every even number greater than or equal to $4\ga$ belongs to ${\rm S}$. So the weight contributed by elements $m \geq 4\ga$ of ${\rm S}$ will be minimized when there are no such {\it odd} elements $m$ strictly less than $\ell_g$. Geometrically, this means that $T_{\rm K}$ stabilizes to a {\it staircase}, i.e a path in which up- and rightward steps alternate. It follows easily $T_{\rm K}$ will be of minimal weight precisely when it {\it is} a staircase for which the $\ga$ even numbers $P_i \in [2,4\ga], 1 \leq i \leq \ga$ that belong to ${\rm S}$ are maximal, namely when
\[
P_i=2\ga+2j, 1\leq j \leq \ga.
\]
Since the first column of $T_{\rm K}$ has precisely $(g-1)-(2\ga+2-1)=g-2\ga-2$ boxes, we deduce that its total weight is
\[
W(T_{\rm K})= \binom{g-2\ga-1}{2}
\]
and it follows that
\[
W_{\rm K}= W(T_{\rm K})+ g-1= {g-2\ga \choose 2}+ 2\ga.
\]
It is clear, moreover, that there are $\ga$-hyperelliptic semigroups of genus $g$ whose ${\rm K}$-weights realize the minimum value of $W_{\rm K}$.

\begin{rem}
Strictly speaking, the preceding geometric argument requires $g \geq 2\ga+2$. However, it is easy to check that when $g=2\ga+1$, the minimal ${\rm K}$-weight is realized by a $\ga$-hyperelliptic semigroup with an empty ${\rm K}$-tableau $T_{\rm K}$. The ${\rm K}$-weight of the corresponding semigroup is then $W_{\rm K}=g-1=2\ga$, as desired. The same argument shows that when $g=2\ga$, the minimal ${\rm K}$-weight is $g-1=2\ga-1$. So our lower bound on $g$ is sharp.
\end{rem}

\medskip
We will now argue that the maximum possible $\ga$-hyperelliptic ${\rm K}$-weight is achieved precisely when
\[
{\rm S}={\rm S}_0:= \langle 4, 4\ga+2, 2g-4\ga+1 \rangle 
\]
just as is the case for ${\rm S}$-weights \cite{To2}. See Figure 1 for the ${\rm K}$-tableau associated with ${\rm S}_0$ when $\ga=3$ and $g=20$.

\begin{figure}

\begin{tikzpicture}[scale=0.30]
\draw[blue, very thin] (0,0) rectangle (20,20);
\filldraw[draw=blue, fill=red] (0,3) rectangle (1,4);

\filldraw[draw=blue, fill=red] (0,4) rectangle (1,5);

\filldraw[draw=blue, fill=red] (0,5) rectangle (1,6);

\filldraw[draw=blue, fill=red] (0,6) rectangle (1,7);
\filldraw[draw=blue, fill=red] (1,6) rectangle (2,7);

\filldraw[draw=blue, fill=lightgray] (0,7) rectangle (1,8);
\filldraw[draw=blue, fill=red] (1,7) rectangle (2,8);

\filldraw[draw=blue, fill=lightgray] (0,8) rectangle (1,9);
\filldraw[draw=blue, fill=lightgray] (1,8) rectangle (2,9);

\filldraw[draw=blue, fill=lightgray] (0,9) rectangle (1,10);
\filldraw[draw=blue, fill=lightgray] (1,9) rectangle (2,10);
\filldraw[draw=blue, fill=lightgray] (2,9) rectangle (3,10);

\filldraw[draw=blue, fill=lightgray] (0,10) rectangle (1,11);
\filldraw[draw=blue, fill=lightgray] (1,10) rectangle (2,11);
\filldraw[draw=blue, fill=lightgray] (2,10) rectangle (3,11);
\filldraw[draw=blue, fill=lightgray] (3,10) rectangle (4,11);

\filldraw[draw=blue, fill=lightgray] (0,11) rectangle (1,12);
\filldraw[draw=blue, fill=lightgray] (1,11) rectangle (2,12);
\filldraw[draw=blue, fill=lightgray] (2,11) rectangle (3,12);
\filldraw[draw=blue, fill=lightgray] (3,11) rectangle (4,12);
\filldraw[draw=blue, fill=lightgray] (4,11) rectangle (5,12);

\filldraw[draw=blue, fill=lightgray] (0,12) rectangle (1,13);
\filldraw[draw=blue, fill=lightgray] (1,12) rectangle (2,13);
\filldraw[draw=blue, fill=lightgray] (2,12) rectangle (3,13);
\filldraw[draw=blue, fill=lightgray] (3,12) rectangle (4,13);
\filldraw[draw=blue, fill=lightgray] (4,12) rectangle (5,13);
\filldraw[draw=blue, fill=lightgray] (5,12) rectangle (6,13);

\filldraw[draw=blue, fill=lightgray] (0,13) rectangle (1,14);
\filldraw[draw=blue, fill=lightgray] (1,13) rectangle (2,14);
\filldraw[draw=blue, fill=lightgray] (2,13) rectangle (3,14);
\filldraw[draw=blue, fill=lightgray] (3,13) rectangle (4,14);
\filldraw[draw=blue, fill=lightgray] (4,13) rectangle (5,14);
\filldraw[draw=blue, fill=lightgray] (5,13) rectangle (6,14);
\filldraw[draw=blue, fill=lightgray] (6,13) rectangle (7,14);

\filldraw[draw=blue, fill=lightgray] (0,14) rectangle (1,15);
\filldraw[draw=blue, fill=lightgray] (1,14) rectangle (2,15);
\filldraw[draw=blue, fill=lightgray] (2,14) rectangle (3,15);
\filldraw[draw=blue, fill=lightgray] (3,14) rectangle (4,15);
\filldraw[draw=blue, fill=lightgray] (4,14) rectangle (5,15);
\filldraw[draw=blue, fill=lightgray] (5,14) rectangle (6,15);
\filldraw[draw=blue, fill=lightgray] (6,14) rectangle (7,15);
\filldraw[draw=blue, fill=lightgray] (7,14) rectangle (8,15);

\filldraw[draw=blue, fill=lightgray] (0,15) rectangle (1,16);
\filldraw[draw=blue, fill=lightgray] (1,15) rectangle (2,16);
\filldraw[draw=blue, fill=lightgray] (2,15) rectangle (3,16);
\filldraw[draw=blue, fill=lightgray] (3,15) rectangle (4,16);
\filldraw[draw=blue, fill=lightgray] (4,15) rectangle (5,16);
\filldraw[draw=blue, fill=lightgray] (5,15) rectangle (6,16);
\filldraw[draw=blue, fill=lightgray] (6,15) rectangle (7,16);
\filldraw[draw=blue, fill=lightgray] (7,15) rectangle (8,16);
\filldraw[draw=blue, fill=lightgray] (8,15) rectangle (9,16);

\filldraw[draw=blue, fill=lightgray] (0,16) rectangle (1,17);
\filldraw[draw=blue, fill=lightgray] (1,16) rectangle (2,17);
\filldraw[draw=blue, fill=lightgray] (2,16) rectangle (3,17);
\filldraw[draw=blue, fill=lightgray] (3,16) rectangle (4,17);
\filldraw[draw=blue, fill=lightgray] (4,16) rectangle (5,17);
\filldraw[draw=blue, fill=lightgray] (5,16) rectangle (6,17);
\filldraw[draw=blue, fill=lightgray] (6,16) rectangle (7,17);
\filldraw[draw=blue, fill=lightgray] (7,16) rectangle (8,17);
\filldraw[draw=blue, fill=lightgray] (8,16) rectangle (9,17);
\filldraw[draw=blue, fill=lightgray] (9,16) rectangle (10,17);

\filldraw[draw=blue, fill=lightgray] (0,17) rectangle (1,18);
\filldraw[draw=blue, fill=lightgray] (1,17) rectangle (2,18);
\filldraw[draw=blue, fill=lightgray] (2,17) rectangle (3,18);
\filldraw[draw=blue, fill=lightgray] (3,17) rectangle (4,18);
\filldraw[draw=blue, fill=lightgray] (4,17) rectangle (5,18);
\filldraw[draw=blue, fill=lightgray] (5,17) rectangle (6,18);
\filldraw[draw=blue, fill=lightgray] (6,17) rectangle (7,18);
\filldraw[draw=blue, fill=lightgray] (7,17) rectangle (8,18);
\filldraw[draw=blue, fill=lightgray] (8,17) rectangle (9,18);
\filldraw[draw=blue, fill=lightgray] (9,17) rectangle (10,18);
\filldraw[draw=blue, fill=lightgray] (10,17) rectangle (11,18);
\filldraw[draw=blue, fill=red] (11,17) rectangle (12,18);
\filldraw[draw=blue, fill=red] (12,17) rectangle (13,18);

\filldraw[draw=blue, fill=lightgray] (0,18) rectangle (1,19);
\filldraw[draw=blue, fill=lightgray] (1,18) rectangle (2,19);
\filldraw[draw=blue, fill=lightgray] (2,18) rectangle (3,19);
\filldraw[draw=blue, fill=lightgray] (3,18) rectangle (4,19);
\filldraw[draw=blue, fill=lightgray] (4,18) rectangle (5,19);
\filldraw[draw=blue, fill=lightgray] (5,18) rectangle (6,19);
\filldraw[draw=blue, fill=lightgray] (6,18) rectangle (7,19);
\filldraw[draw=blue, fill=lightgray] (7,18) rectangle (8,19);
\filldraw[draw=blue, fill=lightgray] (8,18) rectangle (9,19);
\filldraw[draw=blue, fill=lightgray] (9,18) rectangle (10,19);
\filldraw[draw=blue, fill=lightgray] (10,18) rectangle (11,19);
\filldraw[draw=blue, fill=lightgray] (11,18) rectangle (12,19);
\filldraw[draw=blue, fill=red] (12,18) rectangle (13,19);
\filldraw[draw=blue, fill=red] (13,18) rectangle (14,19);
\filldraw[draw=blue, fill=red] (14,18) rectangle (15,19);
\filldraw[draw=blue, fill=red] (15,18) rectangle (16,19);

\end{tikzpicture}

\caption{Tableau $T_{\rm K}$ associated with the weight-maximizing $\ga$-hyperelliptic semigroup ${\rm S}_0=\langle 4, 4\ga+2, 2g-4\ga+1 \rangle$ when $\ga=3$ and $g=20$. The (irrelevant) uppermost line is left empty, while the disparity in weights between the maximizing and minimizing semigroups is in red.}
\end{figure}
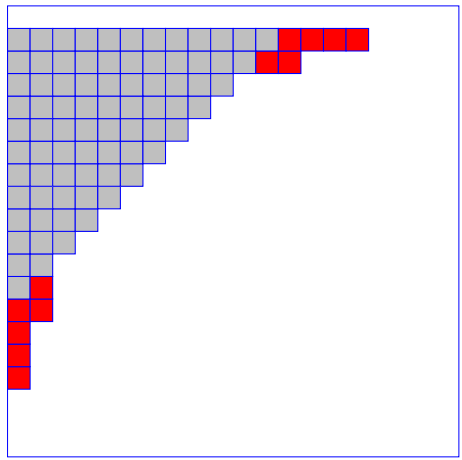

\medskip
Since ${\rm K}$-weight and ${\rm S}$-weight agree for symmetric semigroups, we may (and shall) assume that ${\rm S}$ is nonsymmetric. We will exploit the dual relationship between ramification and weight already used in \cite{To2} to prove that ${\rm S}_0$ is of maximal ${\rm S}$-weight among $\ga$-hyperelliptic semigroups of genus $g$. (Indeed, our argument is a modification of that used in \cite{To2}.)

\medskip
Namely, let $R=R({\rm S})$ denote the {\it total ramification} of ${\rm S}$, given by
\begin{equation}\label{total_ramification}
R:= \sum_{i=1}^g (m_i-i)= \sum_{i=1}^g m_i - \binom{g+1}{2}
\end{equation}
where $m_1 < \cdots < m_g$ are the $g$ smallest nonzero elements of ${\rm S}$. Using the structure theory for $\ga$-hyperelliptic semigroups presented in \cite{To,To2}, we may rewrite the total ramification \eqref{total_ramification} of a $\ga$-hyperelliptic semigroup ${\rm S}$ in the following form:
\[
R= \sum_{i=1}^{\ga} (n_i+u_i)+ \sum_{i=1}^{g-2\ga} (4\ga+2i)- \binom{g+1}{2}
\]
where $n_1 < \cdots < n_{\ga}$ are the smallest nonzero even elements and $u_1 > \cdots > u_{\ga}$ are the odd nonzero elements of ${\rm S}$ less than $2g$, respectively.

\medskip
Maximizing $W_{\rm S}$ is then equivalent to minimizing $R({\rm S})$. Similarly, maximizing $W_{\rm K}$ is equivalent to minimizing 
\begin{equation}\label{K-ramification}
R_{\rm K}:= \sum_{i=1}^{\ga} (n_i+u_i)+ \sum_{i=1}^{g-2\ga-1} (4\ga+2i)- \binom{g}{2}- 2k
\end{equation}
where $k=k({\rm S})$ is the number of odd elements of ${\rm S}$ greater than or equal to the conductor.
Geometrically speaking, \eqref{K-ramification} computes the area of the complement of $T_{\rm K}$ in its minimal $(g-1) \times (g-1)$ bounding box inside of $\Ga$.

\medskip
On the other hand, when ${\rm S}={\rm S}_0$ the sum $\sum_{i=1}^{\ga} n_i$ is minimized, while $k({\rm S}_0)=0$. 
So in light of \eqref{K-ramification}, it suffices to check that
\[
\sum_{i=1}^{k({\rm S})} (u_i({\rm S})- u_i({\rm S}_0))- 2k({\rm S})
\]
is nonnegative for every nonsymmetric $\ga$-hyperelliptic semigroup ${\rm S}$, for which it suffices in turn to show that
\begin{equation}\label{u_i_inequality}
u_i({\rm S})- u_i({\rm S}_0) \geq 2
\end{equation}
for all $1 \leq i \leq k$. The inequality \eqref{u_i_inequality} follows immediately, however, from the fact that ${\rm S}_0$ is symmetric, while ${\rm S}$ is not.
\end{proof}


Finally, just as in \cite{OTV}, it is natural to ask for refinements of Theorem~\ref{ineq_thm}. We have the following ${\rm K}$-weight analogue of \cite[Props. 2.10, 2.12(1)]{OTV}.
\begin{prop}\label{refined_weights}
Let $\ga \geq 1$ be an integer, and let ${\rm S}$ denote a $\ga$-hyperelliptic semigroup of genus $g \geq 3\ga$. Assume that the multiplicity, i.e., the smallest nonzero element of ${\rm S}$ is 4. The ${\rm K}$-weight of ${\rm S}$ then satisfies
\[
W_{\rm K} \in \bigg\{\binom{g-2\ga}{2}+ \ga^2+\ga+ k^2-3k+2: k=1, \dots, \ga+1 \bigg\};
\]
and $W_{\rm K}=\binom{g-2\ga}{2}+ \ga^2+\ga+ k^2-3k+2$ if and only if ${\rm S}=\langle 4, 4\ga+2, 2g-2\ga-2k+3,2g-2\ga+2k+1\rangle$. In particular, every $\ga$-hyperelliptic semigroup with multiplicity $m=4$ of {\it nonmaximal} weight satisfies
\[
\binom{g-2\ga}{2}+ \ga^2+\ga \leq W_{\rm K} \leq \binom{g-2\ga}{2}+ 2(\ga^2-\ga)+2.
\]
\end{prop}

\begin{proof}
Our result is a consequence of \cite[Prop. 2.10]{OTV}, which establishes that
\[
W_{\rm S} \in \bigg\{\binom{g-2\ga}{2}+ \ga^2-\ga+ k^2-k: k=1, \dots, \ga+1 \bigg\};
\]
and $W_{\rm S}=\binom{g-2\ga}{2}+ \ga^2-\ga+ k^2-k$ if and only if ${\rm S}=\langle 4, 4\ga+2, 2g-2\ga-2k+3,2g-2\ga+2k+1\rangle$. Indeed, letting ${\rm S}_0$ denote the latter semigroup, we have
\[
{\rm S}_0= 2 \langle 2, 2\ga+1 \rangle \sqcup \{2g-2\ga-2k+3, 2g-2\ga-2k+7, \dots, 2g-2\ga+2k-1, 2g-2\ga+2k+1, \dots\};
\]
in particular, the largest gap of ${\rm S}_0$ is $\ell_g=2g-2\ga+2k-3$. We conclude immediately using the general fact that $W_{\rm K}= W_{\rm S}+ 2g-1-\ell_g$.
\end{proof}
It would be interesting to extend the reach of Proposition~\ref{refined_weights} to higher multiplicities, and e.g., to establish an analogue of \cite[Prop. 2.12(2)]{OTV}, which establishes an upper bound on the ${\rm S}$-weight when $m \geq 6$. We leave this for future work.

\end{document}